\documentclass[12pt]{amsart}
\usepackage{amssymb,amsmath,amsfonts}
\usepackage{mathrsfs}
\usepackage{ae,aecompl}
\usepackage[T1]{fontenc}
\usepackage[pdftex]{graphicx}

\renewcommand{\mathcal}{\mathscr}
\renewcommand{\epsilon}{\varepsilon}
\newcommand{\dimH}{\dim_{\mathcal H}}

\newtheorem{theorem}{Theorem}[section]
\newtheorem{lemma}[theorem]{Lemma}

\newtheorem{proposition}[theorem]{Proposition}
\newtheorem{conjecture}[theorem]{Conjecture}
\newtheorem{question}[theorem]{Question}

\theoremstyle{definition}

\newtheorem{example}[theorem]{Example}

\theoremstyle{remark}
\newtheorem{remark}[theorem]{Remark}
\newtheorem{observation}[theorem]{Observation}
\newtheorem*{main}{Proof of Theorem \ref{main}}

\numberwithin{equation}{section}

\newcommand{\ba}{{\bf a}}
\newcommand{\bi}{{\bf i}}
\newcommand{\bj}{{\bf j}}
\newcommand{\bk}{{\bf k}}
\newcommand{\bq}{{\bf q}}
\newcommand{\CO}{{\mathcal O}}
\newcommand{\bJ}{{\bf J}}
\newcommand{\mT}{{\bf T}}
\newcommand{\mR}{{\bf r}}
\newcommand{\mQ}{{\bf Q}}
\begin{document}

\title{Dimension of uniformly random self-similar fractals}
\author{Henna Koivusalo}
\date{\today}
\thanks{Research supported by Academy of Finland, the Centre of Excellence in Analysis and Dynamics Research, and Jenny and Antti Wihuri Foundation. }

\begin{abstract}
We calculate the almost sure Hausdorff dimension of uniformly random self-similar fractals. These random fractals are generated from a finite family of similarities, where the linear parts of the mappings are independent uniformly distributed random variables at each step of iteration. We also prove that the Lebesgue measure of such sets is almost surely positive in some cases. 
\end{abstract}
\maketitle

\section{Introduction}

The systematic study of iterated function systems (IFSs) and the corresponding fractal sets was originated by Hutchinson in \cite{Hutchinson81}. He proved that, given a collection of similarities $\{S_1,\dots,S_m\}$ of contraction ratios $\{r_1,\dots,r_m\}$, satisfying the open set condition, the unique nonempty, compact invariant set $F$ has dimension equal to the solution $s$ of $\sum_{i=1}^mr_i^s=1$. The open set condition was later proved to be equivalent to the positivity of the $s$-dimensional Hausdorff measure of $F$ by Schief \cite{Schief94}. The corresponding result for self-affine fractals, that is, sets produced from a collection of affine contractions, is due to Falconer \cite{Falconer88} and Solomyak \cite{Solomyak98}. Their result holds for almost all choices of translations. There is a class of self-affine sets, the Bedford-McMullen carpets, which usually does not satisfy the dimension formula of Solomyak and Falconer even though the open set conditions holds, and thus gives examples proving sharpness of their result. The dimension of Bedford-McMullen carpets is calculated in \cite{Bedford84} and \cite{McMullen84}. 

There are several ways to randomize the construction of IFS fractals. We mention a few relevant examples, but the list is not meant to be exhaustive. The dimension formula for random self-similar fractals was obtained, independently around the same time by Falconer \cite{Falconer86}, Graf \cite{Graf87}, and Mauldin and Williams \cite{MauldinWilliams86}. Later on Graf, Mauldin and Williams \cite{GrafMauldinWilliams88} discovered the gauge function giving positive and finite Hausdorff measure to the random self-similar fractal in its dimension. Dimensional properties of the percolation model were studied by Falconer and Grimmett in \cite{FalconerGrimmett92}. Falconer and Miao \cite{FalconerMiao10} calculated the dimensions of random subsets of self-affine sets, Jordan, Pollicott and Simon \cite{JordanPollicottSimon07} those of randomly perturbed self-affine sets, and Gatzouras and Lalley \cite{GatzourasLalley94} those of random Bedford-McMullen carpets. In all of these models the choice for the IFS is done independently and using the same distribution at each step of the construction. We point out that other types of probability measures have also been studied, see J\"arvenp\"a\"a et al. \cite{JJKKSS} and Barnsley, Hutchinson and Stenflo \cite{BarnsleyHutchinsonStenflo05}, \cite{BarnsleyHutchinsonStenflo12}, for example. 

Dimensional results on random self-similar fractals often require some type of non-overlapping condition, see Barnsley, Hutchinson and Stenflo \cite[(2.3)]{BarnsleyHutchinsonStenflo12}, or Falconer \cite[(7.9)]{Falconer86}, Mauldin and Williams \cite[(2), section 1]{MauldinWilliams86}, or Graf \cite[Theorem 7.6, condition b)]{Graf87}. In contrast, in the current work we do not assume a step-by-step separation condition. We study a class of random self-similar fractals, which we call uniformly random self-similar sets, meaning that the linear parts of generating similitudes are uniformly distributed at each step of the construction, and independent of each other. The translations are fixed to be different but are otherwise arbitrary. 

Another model of random similitude IFSs with fixed translations and uniformly distributed linear parts has been considered by Peres, Simon and Solomyak \cite{PeresSimonSolomyak06}. They studied general problems related to the absolute continuity and dimension of random projections of Bernoulli type measures to the real line. Their results also imply a dimension result for a class of random self-similar sets on the real line (see \cite[Corollary 2.5]{PeresSimonSolomyak06}). In their model independent, absolutely continuous, multiplicative errors to the IFS are introduced at each level of construction. Their probability structure significantly differs from ours, since in our model the linear parts of the mappings are independent both between levels and inside them. 

We calculate the almost sure Hausdorff dimension of a uniformly random self-similar set. It is the minimum of a solution $s$ of an expectation equation (see Lemma \ref{number}), and $d$, the dimension of the space. Further, we prove that when $s>d$, the set has almost surely positive Lebesgue measure. The method of proof has been extracted from the dimension theory of self-affine sets: The dimension bounds are obtained from energy estimates, following ideas of Falconer \cite{Falconer88}, and a key lemma is to prove that a transversality condition, such as in \cite[formula (26)]{JordanPollicottSimon07}, holds (see Lemma \ref{transversality}). 

The paper is organized as follows: First in Section \ref{preliminaries} we introduce the notation used, and formulate the main theorem, Theorem \ref{main}. Section \ref{beginning} concerns the geometric properties of uniformly random sets. In Section \ref{end} we give the energy estimate and deduce the main theorem from it. We shortly discuss two related conjectures in Section \ref{further}.

\section{Preliminaries}\label{preliminaries}

We begin by defining $\Omega$, the space of labelled trees. Fix vectors $a_1, \dots, a_m\in\mathbb R^d$, $a_i\neq a_j$ for $i\neq j$, and denote  $\min_{i\neq j} |a_i-a_j| =a_-$. Fix numbers $0<\sigma_-\leq \sigma_+<1$. Let $\bJ_k=\{1,\dots,m\}^k, \bJ_\infty = \{1,\dots, m\}^{\mathbb N}$, and $\bJ=\cup _{k=1}^\infty \bJ_k$. Denote the orthogonal group of $\mathbb R^d$ by $\CO(d)$. Let $\omega:\bJ\to]\sigma_-,
\sigma_+[\times\CO(d)$ be an $m$-branching tree, edges of which are labelled by $]\sigma_-,\sigma_+[\times\CO(d)$. Denote the space of all this kind of labelled trees by $\Omega$. 

Next we will give a probability measure $\mathbb P$ on $\Omega$. Let $\theta$ be the unique uniformly distributed probability on $\CO (d)$ (that is, $\theta$ is the Haar measure, see \cite[Chapter XI]{Halmos74}). Notice that $\theta$ has the property that for $A\subset S^{d-1}$, $x\in S^{d-1}$,
\begin{equation}\label{theta}
\theta\{g\in\mathcal O(d)\mid g(x)\in A\}=\sigma_{d-1}(A), 
\end{equation}
where $\sigma_{d-1}$ is the normalized surface measure on $S^{d-1}$ (see \cite[Theorem 3.7]{Mattila95}). Let $\lambda$ be the normalized Lebesgue measure on $]\sigma_-,\sigma_+[$. Taking the product $(\lambda\times\theta)^{\bJ}$ over the tree defines a probability measure $\mathbb P$ on $\Omega$. Denote the mapping $\omega(\bi)\in]\sigma_-,\sigma_+[\times\CO(d)$ by $T_\bi^\omega=r_\bi^\omega Q_\bi^\omega$. Notice that for $\bi\neq\bj\in\bJ$ the labels $T_\bi^\omega$ and $T_\bj^\omega$ are independent with respect to $\mathbb P$. Denote the expectation by $\mathbb E$. 

Now we are able to make precise the notion of uniformly random self-similar sets. For $\bi\in\bJ$ or $\bJ_\infty$, denote by $\bi_k$ the initial word of $\bi$ of length $k$ and by $i_k$ the $k$-th symbol in $\bi$. Put $f^\omega_{\bj}(x)=T^\omega_\bj(x) + a_{j_{|\bj|}}$ for all $\bj\in\bJ$ and $x\in \mathbb R^d$, and let $\pi:
{\bf J}_\infty\times \Omega \to\mathbb R^d$ be the mapping
\begin{align*}
\pi(\bi, \omega)&=\lim_{k\to\infty} \left( a_{i_1} + T_{\bi_1}^\omega(a_{i_2}) + \dots + T_{\bi_1}^\omega T_{\bi_2}^\omega\cdots T_{\bi_k}^\omega(a_{i_{k+1}})\right )\\
&=\lim_{k\to\infty}f^\omega_{\bi_1}\circ\dots\circ f^\omega_{\bi_k}(0),
\end{align*}
and define the {\it uniformly random self-similar fractal} $F(\omega)$ as 
\begin{equation}\label{defF}
F(\omega) = \bigcup_{\bi\in {\bf J}_\infty}\pi(\bi, \omega).
\end{equation}

Next we introduce some more notation related to the sequence space $\bJ$. Denote by $|\bi|$ the length, or the number of indices of $\bi$. If $\bi$ and $\bj$ are finite or infinite words such that $\bi_{|\bi|}=\bj_{|\bi|}$, then write $\bi\leq\bj$. If neither $\bi\leq\bj$ nor $\bj\leq\bi$, we say that $\bi$ and $\bj$ are incomparable and write $\bi\bot\bj$. Let $\bi\wedge\bj$ be the word of maximal length such that $\bi\wedge\bj\leq\bi$ and $\bi\wedge\bj\leq\bj$. Notice that $\bi\wedge\bj$ can be empty. For $\bi\in \bJ$, let
\[
[\bi]=\{\bj\in\bJ_\infty\mid \bj_{|\bi|}=\bi\}.
\]
For the sake of brevity, for every $\bi\in\bJ$, let 
\[
\begin{aligned}
\mR_\bi^\omega &=r_{\bi_1}^\omega r_{\bi_2}^\omega\cdots r_{\bi_{|\bi|}}^\omega,
\mQ_\bi^\omega=Q_{\bi_1}^\omega\circ Q_{\bi_2}^\omega\circ\dots\circ Q_{\bi_{|\bi|}}^\omega\\
\textrm{and }\mT_\bi^\omega &=T_{\bi_1}^\omega\circ T_{\bi_2}^\omega\circ\dots\circ T_{\bi_{|\bi|}}^\omega.
\end{aligned}
\]
When there is no threat of misunderstanding, we may suppress the relation to 
$\omega$. 

The proof of the following simple lemma is standard, see \cite{Falconer86} for instance, but we give a short proof for the sake of completeness. 

\begin{lemma}\label{number}
There exists a unique number $s$ satisfying
\begin{equation}\label{defnum}
\mathbb E(\sum_{i=1}^mr_i^s)=1.
\end{equation}
Furthermore, $\mathbb E(\sum_{i=1}^mr_i^t)<1$ for all $t>s$.
\end{lemma}
\begin{proof}
The function $\mathbb E(\sum_{i=1}^mr_i^s)$ is continuous and strictly decreasing in $s$, by dominated convergence and the fact $0<\sigma_-< r_i< \sigma_+<1$. For $s=0$ it attains the value $m$ and decreases to $0$ when $s\to\infty$. Thus a unique value satisfying the equation \eqref{defnum} exists and the latter claim becomes apparent. 
\end{proof}

We now formulate the main theorem, but the proof is postponed until the end of Section \ref{end}. 

\begin{theorem}\label{main}
We have $\dimH F(\omega)=\min\{s,d\}$  almost surely. Furthermore, when $s>d$, the set $F(\omega)$ has positive Lebesgue measure almost surely. 
\end{theorem}

\section{Geometric lemmas}\label{beginning}

In this section we prove that a transversality condition holds for our model. We will need a bit of more notation and begin with an observation. 

\begin{observation}\label{shorterx}
Fix $\bi, \bj\in \bJ_\infty$ with $|\bi\wedge\bj|=p$. Notice that 
\[
\begin{aligned}
|\pi(\bi, \omega)&-\pi(\bj,\omega)|\\
&=\mR_{\bi\wedge\bj}^\omega|\mQ^\omega_{\bi\wedge\bj}(a_{i_{p+1}}-a_{j_{p+1}} + T^\omega_{\bi_{p+1}}(x(\bi,\omega, p))-T^\omega_{\bj_{p+1}}(x(\bj,\omega, p)))|\\
&=\mR^\omega_{\bi\wedge\bj}|a_{i_{p+1}}-a_{j_{p+1}} + T^\omega_{\bi_{p+1}}(x(\bi,\omega, p))-T^\omega_{\bj_{p+1}}(x(\bj,\omega, p))|,
\end{aligned}
\]
where the random variables $x(\bi,\omega, p)\in\mathbb R^d$ and $x(\bj,\omega, p)\in\mathbb R^d$ are independent of each other, since $|\bi\wedge\bj|= p$ and $x(\ba,\omega, p)$ only depends on $T_{\ba_k}^\omega$ for $k>p+1$. 
\end{observation}

Denote by $\mathbb P^\bi$ the probability on node $\bi$, and let $\mathbb P_\bi=(\lambda\times\theta)^{\bJ\setminus \{\bi\}}$. The statement of the following lemma was inspired by \cite[(26)]{JordanPollicottSimon07}, and the proof influenced by \cite[Lemma 5.1]{JordanPollicottSimon07}. 

\begin{lemma}\label{transversality}
The following transversality condition holds: Fix $\bi,\bj\in \bJ_\infty$ with $|\bi\wedge\bj|=p$ and $\rho>0$. Assume $|x(\bi,\omega, p)|\geq |x(\bj,\omega, p)|$. Then 
\[
\mathbb P^{\bi_{p+1}}(|\pi(\bi, \omega)-\pi(\bj,\omega)|<\rho)\leq C'\frac{\rho^{d}}{(\mR_{\bi\wedge\bj})^{d}},
\]
where $C'=C'(a_-, d, \sigma_-)$. 
\end{lemma}
\begin{proof}
Denote $a=a_{i_{p+1}}-a_{j_{p+1}}$, $x= x(\bi,\omega, p)$ and $y=x(\bj,\omega, p)$. Then $|a|\geq a_->0$. Notice that throughout the proof the notions $a, T_{\bj_{p+1}}, x$ and $y$ are fixed, since they don't depend on the label at node $\bi_{p+1}$.
Recalling Observation \ref{shorterx}, the probability we want to estimate is the probability of the event
\[
A=\{T_{\bi_{p+1}}\in]\sigma_-,\sigma_+[\times\mathcal O(d)\mid|a+T_{\bi_{p+1}}(x)-T_{\bj_{p+1}}(y)|< \gamma:=\frac{\rho}{\mR_{\bi\wedge\bj}}\}.
\]
Firstly, if $r_{\bi_{p+1}}\notin [|x|^{-1}(|a-T_{\bj_{p+1}}(y)|-\gamma),|x|^{-1}(|a-T_{\bj_{p+1}}(y)|+\gamma)]=:I$, then 
\begin{align*}
|a+T_{\bi_{p+1}}(x)-T_{\bj_{p+1}}(y)|&\geq ||T_{\bi_{p+1}}(x)|-|a-T_{\bj_{p+1}}(y)||\\
&=|r_{\bi_{p+1}}|x| -|a-T_{\bj_{p+1}}(y)||\\
&> \gamma.
\end{align*}
Here $\lambda(I)=2\gamma|x|^{-1}$. Notice that, since $|x|\geq|y|$, $A=\emptyset$ whenever $|x|<\tfrac 12 (a_--\gamma)$, and we may assume that the opposite inequality holds. Similarly $|a-T_{\bj_{p+1}}(y)|\geq \sigma_-|x|-\gamma$. We now have $\lambda(I)\leq 4\gamma(a_--\gamma)^{-1}$. 

Denote the open ball of radius $\delta$ and centre $z$ by $B(z,\delta)$ and the cone of direction $v$ and opening angle $\alpha$ by $V(v,\alpha)$. Further, let $\beta=\gamma|a-T_{\bj_{p+1}}(y)|^{-1}$. Then, for all $r_{\bi_{p+1}}$, for $\beta<1$, that is, for all $\gamma$ satisfying $\gamma|a-T_{\bj_{p+1}}(y)|^{-1}<1$, we have 
\begin{align*}
\{Q_{\bi_{p+1}}\mid |a+Q_{\bi_{p+1}}&(r_{\bi_{p+1}}x)-T_{\bj_{p+1}}(y)|<\gamma\}\\
&=\{Q_{\bi_{p+1}}\mid Q_{\bi_{p+1}}(r_{\bi_{p+1}}x)\in B(a-T_{\bj_{p+1}}(y),\gamma)\}\\
&\subset \{Q_{\bi_{p+1}}\mid Q_{\bi_{p+1}}(r_{\bi_{p+1}}x)\in V\big(a-T_{\bj_{p+1}}(y), \arcsin\beta\big)\}\\
&=\{Q_{\bi_{p+1}}\mid Q_{\bi_{p+1}}(\frac{x}{|x|})\in V\big(a-T_{\bj_{p+1}}(y), \arcsin\beta\big)\cap S^{d-1}\}\\
&=:V.
\end{align*}
Recall that $|x|\geq \tfrac 12(a_--\gamma)$, and $|a-T_{\bj_{p+1}}(y)|\geq \sigma_-\tfrac 12(a_--\gamma)-\gamma$, so that $\beta <1$ for all $\gamma <\tfrac 15a_-\sigma_-\leq \tfrac 15 a_-$. By elementary geometry, recalling \eqref{theta},
\begin{align*}
\theta(V)&=\sigma_{d-1}\Big(V\big(a-T_{\bj_{p+1}}(y), \arcsin\beta\big)\cap S^{d-1}\Big)\leq C''(d)\beta^{d-1}\\
&\leq  C''(d)\gamma^{d-1}(\tfrac 15\sigma_-a_-)^{-d+1}
\end{align*}
for all $\gamma <\tfrac 15a_-\sigma_-\leq \tfrac 15 a_-$. 

By the above considerations, $A\subset I\times V$, and 
\[
\mathbb P^{\bi_{p+1}}(A)\leq \lambda(I)\theta(V)\leq C''(d)\gamma^d5a_-^{-1}(\tfrac 15\sigma_-a_-)^{-d+1}
\]
whenever $\gamma< \tfrac 15a_-\sigma_-\leq \tfrac 15 a_-$. This proves the claim with the constant $C'=5C''(d)a_-^{-1}(\tfrac 15 \sigma_-a_-)^{-d+1}$ for all $\rho<\tfrac 15\sigma_- a_-\mR_{\bi\wedge\bj}$. Larger $\rho$'s can be dealt with by further increa\-sing $C'$ to satisfy $C'\geq \big(\tfrac 15 \sigma_-a_-\big )^{-d}$, since $\mathbb P$ is a probability measure. 
\end{proof}

The following lemma is a simplification of \cite[Lemma 4.5]{JordanPollicottSimon07}.

\begin{lemma}\label{howto}
Fix $\bi\neq \bj\in \bJ_\infty$ with $|\bi\wedge\bj|=p$, and $t<d$. Assume $|x(\bi,\omega, p)|\geq |x(\bj,\omega, p)|$. Then 
\[
\int_\Omega\frac{d\mathbb P^{\bi_{p+1}}(\omega)}{|\pi(\bi,\omega)-\pi(\bj,\omega)|
^t}\leq C\mR_{\bi\wedge\bj}^{-t},
\]
for some $C=C(t,d,a_-, \sigma_-)$.
\end{lemma}

\begin{proof}
Using first \cite[Theorem 1.15]{Mattila95}, and then Lemma \ref{transversality} for $\rho\leq \mR_{\bi\wedge\bj}$ and the trivial 
estimate for $\rho\geq \mR_{\bi\wedge\bj}$
\[
\begin{aligned}
\int_\Omega\frac{d\mathbb P^{\bi_{p+1}}(\omega)}{|\pi(\bi,\omega)-\pi(\bj,\omega)|
^t}&=t\int_0^\infty\mathbb P^{\bi_{p+1}}(\omega\mid|\pi(\bi,\omega)-\pi(\bj,\omega)|
<\rho)\rho^{-t-1}\,d\rho\\
&\leq C't\int_0^{\mR_{\bi\wedge\bj}}\frac{\rho^{d}}{\mR_{\bi\wedge\bj}^{d}}\rho^{-
t-1}\,d\rho + t\int_{\mR_{\bi\wedge\bj}}^\infty\rho^{-t-1}\,d\rho\\
&\leq (\frac{C'}{d-t} + 1)\mR_{\bi\wedge\bj}^{-t}, 
\end{aligned}
\]
where $C'=C'(a_-, d, \sigma_-)$ from Lemma \ref{transversality}. 
\end{proof}

\section{Proof of the main theorem}\label{end}

Recall the number $s$ from Lemma \ref{number}, and denote the Lebesgue measure on $\mathbb R^d$ by $\lambda^d$. In this section we prove the main theorem, Theorem \ref{main}, namely that almost surely $\dimH F(\omega)=\min\{s, d\}$, and that $\lambda^d(F(\omega))>0$ almost surely when $s>d$. We first prove the upper bound for the dimension as Proposition \ref{martingale}. We then define random measures on $\bJ_\infty$, almost surely projecting onto $F$ as measures of finite energy. The lower bound for the dimension is then an easy consequence of the energy estimate. To prove the statement of positive Lebesgue measure, we study the absolute continuity of these measures. 

Fix $0<t<s$ for the time being. For all $k$, denote by $\mathcal F_k$ the sigma-algebra generated by the random variables $T_\ba$ for all $|\ba|\leq k$. 

The proofs of Lemma \ref{mart}, Proposition \ref{martingale} and Lemma \ref{defmeas} are essentially from the proof of \cite[Theorem 15.1]{Falconer03} (also see \cite{MauldinWilliams86} and \cite{Graf87}). For the convenience of the reader, and since the exposition in \cite{Falconer03} is not overly detailed, we repeat the necessary arguments here. 

\begin{lemma}\label{mart}
For all $u>0$ and $k\in\mathbb N$, 
\[
\mathbb E\big(\sum_{|\bi|=k+1}\mR_\bi^u\big)=\mathbb E\Big(\big(\sum_{i=1}^mr_i^u\big)\Big)^{k+1}.
\]
\end{lemma}
\begin{proof}
Notice that, recalling the definitions of $\mR_\bi$ and $r_\bi$ from Section \ref{preliminaries},
\begin{equation}\label{cond}
\begin{aligned}
\mathbb E\big(\sum_{|\bi i|=k+1}\mR_{\bi i}^u\mid\mathcal F_k\big)&=\mathbb E\Big(\sum_{|\bi|=k}\sum_{i=1}^m\mR_\bi^u r_{\bi i}^u\mid\mathcal F_k\Big)\\
&=\sum_{|\bi|=k}\mR^u_\bi\mathbb E\Big(\sum_{i=1}^m r_{\bi i}^u\Big)\\
&=\sum_{|\bi|=k}\mR^u_\bi\mathbb E\big(\sum_{i=1}^mr_i^u\big).
\end{aligned}
\end{equation}
Iterating this the claim follows for $\mathbb E(\sum_{|\bi|=k+1}\mR_\bi^u)=\mathbb E(\mathbb E(\sum_{|\bi|=k+1}\mR_\bi^u\mid\mathcal F_k))$. 
\end{proof}

\begin{proposition}\label{martingale}
For $\mathbb P$-almost all $\omega\in \Omega$
\[
\dimH F(\omega)\leq \min\{s, d\}.
\]
\end{proposition}
\begin{proof}
Certainly $\dimH F\le d$, and we only have to check that $\dimH F\le s$. 

Define a sequence of random variables $X_k=\sum_{|\bi|=k}\mR_\bi^s$. We have, for all $\omega\in\Omega$, for the Hausdorff measure of $F(\omega)$
\[
\mathcal H^s(F(\omega))\leq R\liminf_{k\to\infty}X_k(\omega),
\]
since $F(\omega)\subset\cup_{\bi\in\bJ_k}\pi([\bi],\omega)$, where the diameter of $\pi([\bi],\omega)$ is bounded from above by $R\mR_\bi$ for a constant $R>0$ independent of $\omega$. We prove that $X_k$ is an $L^2$-bounded martingale with respect to the sequence of sigma-algebras $\mathcal F_k$. Firstly, by the choice of $s$, for any $\bi\in\bJ$,
\[
\mathbb E \Big [\sum_{i=1}^m r_{\bi i}^s\Big ]=1
\]
so that by calculation \eqref{cond} above $\mathbb E(X_{k+1}\mid\mathcal F_k)=X_k$ immediately. Thus $X_k$ is a martingale. Furthermore, 
\[
\begin{aligned}
\mathbb E(X_k^2\mid \mathcal F_{k-1})&=\mathbb E\Big[(\sum_{|\bi i|=k}\mR_{\bi i}^s)^2\mid \mathcal F_{k-1} \Big]\\
&=\mathbb E\Big[\sum_{|\bi|=k-1}\mR_\bi^{2s}\sum_{i=1}^m\sum_{j=1}^m r_{\bi i}^s r_{\bi j}^s\mid\mathcal F_{k-1}\Big]\\ 
&\ + \mathbb E
\Big[\sum_{|\bi|=k-1}\sum_{|\ba|=k-1, 
\ba\neq\bi}\mR_{\bi}^s\mR_\ba^s\sum_{i=1}^m\sum_{j=1}^m r_{\bi i}^s r_{\ba j}^s\mid 
\mathcal F_{k-1}\Big]\\
&=\sum_{|\bi|=k-1}\mR_\bi^{2s}\mathbb E\Big[\sum_{i=1}^m\sum_{j=1}^m r_{\bi i}^s r_{\bi j}^s\Big] + \sum_{|\bi|=k-1}\sum_{|\ba|=k-1, \ba\neq \bi}\mR_\bi^s\mR_\ba^s\mathbb E\Big[\sum_{i=1}^m\sum_{j=1}^m r_{\bi i}^s r_{\ba j}^s\Big].
\end{aligned}
\]
For $\ba\neq\bi$, $|\ba|=|\bi|$, the random variables $r_{\bi i}$ and $r_{\ba j}$ are independent for all $i, j=1,\dots,m$. Thus, by choice of $s$, 
\[
\mathbb E\Big [\sum_{i=1}^m\sum_{j=1}^m r_{\bi i}^s r_{\ba j}^s \Big ] = \mathbb E\Big [\sum_{i=1}^m r_{\bi i}^s \Big ]\mathbb E\Big [ \sum_{j=1}^m r_{\ba j}^s\Big ]=1.
\]
Notice that $\mathbb E(\sum_{i=1}^m r_{\bi i}^{2s})$ does not depend on $\bi\in\bJ$. Denote this quantity by $\lambda$, and notice that $\lambda<1$ by Lemma \ref{number}. Then, again by the definition of $s$,
\[
\begin{aligned}
\mathbb E \Big [\sum_{i=1}^m\sum_{j=1}^m r_{\bi i}^s r_{\bi j}^s\Big ] &= \mathbb 
E\Big [\sum_{i=1}^m r_{\bi i}^{2s}\Big ] + \mathbb 
E\Big [\sum_{i=1}^m\sum_{j=1, j\neq i}^m r_{\bi i}^s r_{\bi j}^s\Big ]\\
&\leq\lambda + 1.
\end{aligned}
\]
From the above calculations
\[
\mathbb E(X_k^2\mid \mathcal F_{k-1})\leq \sum_{|\bi| = k-1}\mR_\bi^{2s} \lambda + X_{k-1}^2.
\]
By Lemma \ref{mart}, 
\[
\mathbb E(\sum_{|\bi|=k-1}\mR_{\bi}^{2s})=\lambda^{k-1},
\]
and hence,
\[
\mathbb E(X_k^2)\leq \lambda^k + \mathbb E(X_{k-1}^2)\leq \sum_{k=1}^\infty \lambda^k + 1<\infty.
\]
By the martingale convergence theorem, see \cite[Theorems 12.24 and 12.28]{MortersPeres10}, the $L^2$-boundedness of the martingale $(X_k)$ implies that the sequence of random variables converges (almost surely and in $L^2$) to a random variable $X$ and also that 
\[
\mathbb E(X\mid \mathcal F_k)=X_k,
\]
most importantly giving $\mathbb E(X)=1$. This means that $X(\omega)<\infty$ for almost every $\omega$, and thus the upper bound for the dimension follows. 
\end{proof}

\begin{lemma}\label{defmeas}
There exists a random measure $\mu^\omega$ on $\bJ_\infty$ having the properties
\begin{enumerate}
\item{almost surely $0<\mu^\omega(\bJ_\infty)<\infty$,}
\item{$\mathbb E(\mu^\omega[\bi]\mid\mathcal F_k)=\mR_\bi^{s}$ for all $\bi\in\bJ_k$, and}
\item{for all $k$, $\mathbb E(\sum_{|\bi|=k}\mu^\omega[\bi])=1$.}
\end{enumerate}
\end{lemma}
\begin{proof}
For $\bi\in \bJ$, define a sequence of random variables 
\[
\mu_k[\bi]=\sum_{\bj\in[\bi]\cap\bJ_k}\mR_\bj^s.
\]
Exactly the same proof as above for $X_k$ shows that also $\mu_k[\bi]$ is an $L^2$-bounded martingale, and hence converges to a $\tilde\mu[\bi]$ with $0\leq \tilde\mu[\bi]<\infty$ almost everywhere, and 
\[
\mathbb E(\tilde\mu[\bi]\mid\mathcal F_{|\bi|})=\mR^s_\bi.
\]
Furthermore, since $\tilde\mu[\bi]=\sum_{i=1}^m\tilde\mu[\bi i]$ for all $\bi\in\bJ$, almost surely the cylinder function $\tilde\mu$ extends naturally to a Borel measure $\mu^\omega$ on $\bJ_\infty$ with $\mu^\omega[\bi]=\tilde\mu[\bi]$ for all $\bi\in\bJ$. 

Now, $\tilde\mu[\bi]=0$ with probability $q<1$ and, on the other hand, $\tilde\mu[\bi]=0$ if and only if $\tilde\mu[\bi i]=0$ for all $i=1,\dots,m$. Notice that by self-repeating nature of the probability, $\tilde\mu[\bi]$ and $r_{\bi i}^{-s}\tilde\mu[\bi i]$ have the same distribution. By independence of $\tilde\mu[\bi i]=0$ and $\tilde \mu[\bi j]=0$ for $i\neq j$, this leads to $q^m=q$ and hence $\tilde\mu[\bi]>0$ almost surely. Then $0<\mu^\omega(\bJ_\infty)<\infty$. 

Lemma \ref{mart} and the definition of $s$ give the last claim, since for all $\bi\in\bJ$ we have $\mathbb E(\mu^\omega[\bi])=\mathbb E(\mathbb E(\mu^\omega[\bi]\mid \mathcal F_{|\bi|}))=\mathbb E(\mR_\bi^s)$. 
\end{proof}

The following easy lemma will be the key to proving an energy estimate for the measure $\mu^\omega$. 

\begin{lemma}\label{changemeas}
Let $\omega,\omega'\in\Omega$ and $\bi\in\bJ$. If $T_\ba^\omega=T_\ba^ {\omega'}$ for all $\ba\neq \bi$, then 
\begin{equation}\label{muprop}
\mu^\omega|_{[\bi]} = \left (\frac{r_\bi^\omega}{r_\bi^{\omega'}}\right )^{s}\mu^{\omega'}|_{[\bi]}
\end{equation}
and for all $\bj\bot\bi$, in fact $\mu^\omega|_{[\bj]}=\mu^{\omega'}|_{[\bj]}$. 
\end{lemma}
\begin{proof}
For all $\ba$ with $\bi\leq \ba$ we have
\[
\mR_\ba^\omega=\frac{r_\bi^\omega}{r_\bi^{\omega'}}\mR_\ba^{\omega'},
\]
and for all $\ba$ with $\bi\bot\bj\leq\ba$, we have $\mR_\ba^\omega=\mR_\ba^{\omega'}$. By definition of the measures $\mu^\omega$ and $\mu^{\omega'}$ the claim follows. 
\end{proof}

Denote by $I_t(\nu)$ the $t$-energy of a measure $\nu$ with support $E$, that is, let 
\[
I_t(\nu)=\iint_{E\times E}|x-y|^{-t}\,d\nu(x)\,d\nu (y).
\]
We verify that the expectation of $I_t(\pi_*\mu^\omega)$ is finite for all $t<s$ as Theorem \ref{energy estimate}. Here the image of the measure $\mu^\omega$ under $\pi(\cdot,\omega)$ is denoted by $\pi_*\mu^\omega$. For properties of energies of measures, including their connection to the dimension of the supporting set, see \cite[Chapter 8]{Mattila95}. 

\begin{remark}
By Lemma \ref{defmeas}, for all $\bi\in\bJ$ the function $\omega\mapsto \mu^\omega[\bi]$ is a measurable function. Since all open sets of $\bJ_\infty$ are disjoint finite unions of cylinder sets, also $\omega\mapsto\mu^\omega(A)$ is measurable for all open and closed sets. Since all continuous functions $f$ on $\bJ_\infty$ are limits of sequences of simple functions of the form $\sum_{i=1}^n c_i\chi_{A_i}$ for characteristic functions of open and closed sets $A_i$, also $\omega\mapsto\int f\,d\mu^\omega$ is measurable. Presenting $|\pi(\bi,\omega)-\pi(\bj,\omega)|^{-t}$ as a limit $\lim_{k\to\infty}\min\{|\pi(\bi,\omega)-\pi(\bj,\omega)|^{-t}, k\}$ of continuous functions, we see that $\omega\mapsto I_t(\mu^\omega)$ is measurable. 
\end{remark}

The proof of the following theorem uses some ideas in \cite[Theorem 5.1]{Solomyak98}.

\begin{theorem}\label{energy estimate}
The measure $\pi_*\mu^\omega$ satisfies:
\begin{enumerate}
\item{The expectation of the $t$-energy of $\pi_*\mu^\omega$ is finite for all $t< \min\{s, d\}$.}
\item{If $s>d$, then $\pi_*\mu^\omega$ is almost surely absolutely continuous.}
\end{enumerate}
\end{theorem}

\begin{proof}
(1) Denote by $v(s)$ the number $(\tfrac{\sigma_+}{\sigma_-})^s$. 
By Lemma \ref{defmeas} the measure $\mu^\omega$ is well-defined for almost all $\omega\in\Omega$. Below we only consider $\omega$'s which are typical in this sense. Since $\mu^\omega$'s don't have atoms, we have for the expectation of the energy,
\begin{equation}\label{eq1}
\begin{aligned}
\mathbb E(I_t(\pi_*\mu^\omega))=\iiint&\frac{d\mu^\omega(\bi)\,d\mu^\omega(\bj)\,d\mathbb P(\omega)}{|\pi( 
\bi,\omega)-\pi( \bj,\omega)|
^t}=\iiint_{\bi\neq\bj}\frac{d\mu^\omega(\bi)\,d\mu^\omega(\bj)\,d\mathbb P(\omega)}{|
\pi(\bi,\omega)-\pi(\bj,\omega)|^t}\\
&\leq \sum_{k=0}^\infty\sum_{|\bq|=k}\sum_{i\neq j}\iint_{[\bq j]}\int_{[\bq 
i]}\frac{d\mu^\omega(\bi)\,d\mu^\omega(\bj)\,d\mathbb P(\omega)}{|\pi( 
\bi,\omega)-\pi(\bj,\omega)|^t}.
\end{aligned}
\end{equation}
For a while, fix $k$, $|\bq| = k$ and $i\neq j\in\{1,\dots,m\}$, and furthermore, fix 
$T_\ba^\omega$ for all $\ba\neq \bq i$ and all $\ba\neq \bq j$. Notice that then, given $\bi\in [\bq i]$ 
and $\bj\in[\bq j]$, the vectors $x(\bi, \omega, k)$ and $x(\bj, \omega, k)$ from 
Observation \ref{shorterx} are fixed. Let $X=X(\omega)=\{(\bi,\bj)\in[\bq i]\times[\bq j]\mid |x(\bi,\omega, k)|\geq |x(\bj, \omega, k)|\}$. Let $T_0\in]\sigma_-,\sigma_+[\times\CO(d)$, and denote by $\omega_{\bq i}$ the modification of $\omega\in\Omega$ obtained by changing $T_{\bq i}^{\omega_{\bq i}}=T_0$. Furthermore, by Lemma \ref{changemeas}, Fubini and Lemma \ref{howto}
\[
\begin{aligned}
\iiint_X&\frac{d\mu^\omega(\bi)\,d\mu^\omega(\bj)\,d\mathbb P^{\bq i}(\omega)}{|\pi( \bi,\omega)-\pi(\bj,\omega)|^t}\leq \iiint_{X}\frac{v(s) d\mu^{\omega_{\bq i}}(\bi)\,d\mu^{\omega}|_{[\bq j]}(\bj)\,d\mathbb P^{\bq i}(\omega)}{|\pi(\bi,\omega)-\pi(\bj,\omega)|^t}\\
&= \iint_{X}\int\frac{v(s)d\mathbb P^{\bq i}(\omega)d\mu^{\omega_{\bq i}}(\bi)\,d\mu^\omega|_{[\bq j]}(\bj)}{|\pi(\bi,\omega)-\pi(\bj,\omega)|^t}\\
&\leq v(s)C\mR_\bq^{-t}\mu^{\omega_{\bq i}}[\bq j]\mu^{\omega_{\bq i}}[\bq i].
\end{aligned}
\]
Since $\mu^{\omega_{\bq i}}$ and $\mu^{\omega_{\bq i}}|_{[\bq j]}$ do not depend on $T_{\bq i}^\omega$, by Lemma \ref{changemeas},
\begin{align*}
\int v(s)C\mR_\bq^{-t}\mu^{\omega_{\bq i}}[\bq j]\mu^{\omega_{\bq i}}[\bq i]\,d\mathbb P_{\bq i}(\omega)&=\int v(s)C\mR_\bq^{-t}\mu^{\omega_{\bq i}}[\bq j]\mu^{\omega_{\bq i}}[\bq i]\,d\mathbb P(\omega)\\
&\leq\int v(s)^2C\mR_\bq^{-t}\mu^{\omega}[\bq j]\mu^{\omega}[\bq i]\,d\mathbb P(\omega).
\end{align*}
Furthermore, using the fact that $\mu^{\omega}[\bq i]$ and $\mu^\omega[\bq j]$ are independent when conditioned on $\mathcal F_k$, by Lemma \ref{defmeas}
\begin{equation*}
\begin{aligned}
v(s)^2C\mathbb E(\mathbb E(\mR_\bq^{-t}\mu^{\omega}[\bq i]\mu^\omega[\bq j]\mid\mathcal F_k)) &= v(s)^2C\mathbb E(\mR_\bq^{-t}\mathbb E(\mu^{\omega}[\bq i]\mid\mathcal F_k)\mathbb E(\mu^\omega[\bq j]\mid\mathcal F_k))\\
&\leq v(s)^2C\mathbb E(\mR_\bq ^{-t}\mR^s_{\bq}\mathbb E(\mu^\omega[\bq j]\mid\mathcal F_k))\\
&\leq v(s)^2C\sigma_+^{k(s-t)}\mathbb E(\mu^\omega[\bq j]).
\end{aligned}
\end{equation*}
Combining the above calculations gives
\begin{equation}\label{eq3}
\int\iint_X\frac{d\mu^\omega(\bi)\,d\mu^\omega(\bj)}{|\pi( 
\bi,\omega)-\pi(\bj,\omega)|^t}\,d\mathbb P(\omega)\leq v(s)^2C\sigma_+^{k(s-t)}\mathbb E(\mu^\omega[\bq j]).
\end{equation}
Using \eqref{eq1}, \eqref{eq3}, the counterpart of \eqref{eq3} for $Y=Y(\omega)=\{(\bi,\bj)\in[\bq i]\times[\bq j]\mid|x(\bi,\omega, k)|<|x(\bj,\omega, k)|\}$, and Lemma \ref{defmeas} (3), we obtain 
\begin{equation}\label{sums}
\begin{aligned}
\mathbb E(I_t(&\pi_*\mu^\omega))\\
&=\sum_{k=0}^\infty\sum_{|\bq|=k}\sum_{i\neq j}\int\iint_X\frac{d\mu^\omega(\bi)\,d\mu^\omega(\bj)}{|\pi( 
\bi,\omega)-\pi(\bj,\omega)|^t} + \iint_Y\frac{d\mu^\omega(\bi)\,d\mu^\omega(\bj)}{|\pi( 
\bi,\omega)-\pi(\bj,\omega)|^t}\,d\mathbb P(\omega)\\
&\leq\sum_{k=0}^\infty \sum_{|\bq|=k} 2 m v(s)^2C\sigma_+^{k(s-t)}\mathbb E(\mu^\omega[\bq])\\
&=\sum_{k=0}^\infty 2 m v(s)^2C\sigma_+^{k(s-t)}<\infty,
\end{aligned}
\end{equation}
where the sum converges since $\sigma_+<1$.

(2) To prove the latter claim, by \cite[Lemma 2.12 (3)]{Mattila95} it suffices to check that 
\[
\iint \liminf _{\rho\to 0} \frac{\pi_*\mu^\omega(B(x,\rho))}{\lambda^d(B(x,\rho))}\,d(\pi_*\mu^\omega)(x)\,d\mathbb P (\omega)<\infty,
\]
and thus by Fatou lemma to prove 
\begin{equation*}
\liminf _{\rho\to 0} \rho^{-d}\iint \pi_*\mu^\omega(B(x,\rho))\,d(\pi_*\mu^\omega)(x)\,d\mathbb P<\infty.
\end{equation*}
Here $B(x,\rho)$ is an open ball of center $x$ and radius $\rho$. 

Denote the characteristic function of a set $A$ by $\chi_A$. Notice that 
\begin{align*}
\int \pi_*\mu^\omega(B(x,\rho))&\,d(\pi_* \mu^\omega)(x)=\iint \chi_{B(x,\rho)}(y)\, d(\pi_*\mu^\omega)(y)\, d(\pi_*\mu^\omega)(x)\\
&=\iint \chi_{\{\bk\mid |\pi(\bk,\omega)-\pi(\bj,\omega)|<\rho\}}(\bi)\,d\mu^\omega(\bi)\,d\mu^\omega(\bj)\\
&\le\sum_{k=0}^\infty\sum_{|\bq|=k}\sum_{i\neq j}\int_{[\bq j]}\int_{[\bq 
i]}\chi_{\{\bk\mid |\pi(\bk,\omega)-\pi(\bj,\omega)|<\rho\}}(\bi)\,d\mu^\omega(\bi)\,d\mu^\omega(\bj).
\end{align*}
Fix $k$, $\bq$ and $i\neq j$, and define $\omega_{\bq i}$ and the set $X$ as above. Then by Fubini and Lemma \ref{changemeas}
\begin{align*}
\iiint _X&\chi_{\{\bk\mid |\pi(\bk,\omega)-\pi(\bj,\omega)|<\rho\}}(\bi)\,d\mu^\omega(\bi)\,d\mu^\omega(\bj)\,d\mathbb P^{\bq i}\\
&\leq v(s)\iint _X\int \chi_{\{\bk\mid |\pi(\bk,\omega)-\pi(\bj,\omega)|<\rho\}}(\bi)\,d\mathbb P^{\bq i}\,d\mu^{\omega_{\bq i}}(\bi)\,d\mu^\omega|_{
[\bq j]}(\bj), 
\end{align*}
where, by \cite[Theorem 1.15]{Mattila95} and Lemma \ref{transversality}
\begin{align*}
\int \chi_{\{\bk\mid |\pi(\bk,\omega)-\pi(\bj,\omega)|<\rho\}}(\bi)\,d\mathbb P^{\bq i}&=\int _0^\infty \mathbb P^{\bq i}\{\omega\mid \chi_{\{\bk\mid |\pi(\bk,\omega)-\pi(\bj,\omega)|<\rho\}}(\bi)\geq t\}\,dt\\
&=\mathbb P^{\bq i}\{\omega\mid |\pi(\bi,\omega)-\pi(\bj,\omega)|<\rho\}\\
&\leq C'\rho^d\mR_{\bq}^{-d}.
\end{align*}
After this the estimate 
\[
\iiint _X\chi_{\{\bk\mid |\pi(\bk,\omega)-\pi(\bj,\omega)|<\rho\}}(\bi)\,d\mu^\omega(\bi)\,d\mu^\omega(\bj)\,d\mathbb P\leq C'v(s)^2\rho^d\sigma_+^{k(s-d)}\mathbb E(\mu^\omega[\bq j])
\]
follows as the inequality \eqref{eq3} above in the proof of (1). Repeating this argument on the set $Y$ and summing up over $k$, $\bq$ and $i\neq j$ gives, as in \eqref{sums},
\begin{align*}
\liminf _{\rho\to 0} \rho^{-d}\iint& \pi_*\mu^\omega(B(x,\rho))\,d(\pi_*\mu^\omega)(x)\,d\mathbb P\\
&\leq \sum_{k=0}^\infty \sum_{|\bq|=k} 2 m v(s)^2C'\sigma_+^{k(s-d)}\mathbb E(\mu^\omega[\bq])<\infty,
\end{align*}
finishing the proof. 
\end{proof}

\begin{main}
To prove the first claim, notice that $\dimH F(\omega)\leq \min\{s,d\}$ almost surely by Lemma \ref{martingale}, and fix $0<t<\min\{s,d\}$. By Lemma \ref{defmeas} and Theorem \ref{energy estimate} (1) there exists a finite Borel measure $\mu^\omega$ on $\bJ_\infty$ with
\begin{equation}
\mathbb E(\iint_{F(\omega)\times F(\omega)}\frac{ d(\pi_*\mu^\omega(x))\,d(\pi_*\mu^\omega(y))}{|x-y|^t})<\infty.
\end{equation}
Thus, almost surely, $\dimH F(\omega)\geq t$. (See \cite[Lemma 5.2]{Falconer88}.) Approaching $s$ along a sequence will result in $\dimH F(\omega)\geq s$, almost surely. 

The latter claim is an immediate consequence of Theorem \ref{energy estimate} (2). 
\qed
\end{main}

\section{An example and further problems}\label{further}

We begin this section by giving an example that shows sharpness of Theorem \ref{main} in a sense. This is an example of a fractal set in $\mathbb R^d$ such that for the generating similitudes contraction ratios are uniformly distributed, but rotations deterministic. The dimension of such a set can be strictly less than the number $\min\{s,d\}$ of Theorem \ref{main}. 
\begin{example}
Fix $m$ so large that $\tfrac 1m <\sigma_-$. Consider a system of $m$ similarities in $\mathbb R$ having the random structure described in Section \ref{preliminaries}, that is, each $r_\bi$ is uniformly distributed in $]\sigma_-,\sigma_+[$. Then for the limiting set $F$
\[
\dimH F\leq 1=\min\{1,s\},
\]
where $s>1$ satisfies $m \sigma_-^{s}=1$. Now, embed this set in $\mathbb R^d$ for $d\geq 2$. Still
\[
\dimH F\leq 1<\min\{d,s\},
\]
and the claim of Theorem \ref{main} does not hold for the set $F$. 
\end{example}

We then present ways of generalizing the result. Instead of similarities one could also consider affine mappings, and conjecture
\begin{conjecture}\label{affine}
The Hausdorff dimension of a uniformly random self-affine set is a constant number almost surely.
\end{conjecture}
Here the uniform distribution on the space of contractive bijective linear mappings is the normalized Lebesgue measure $\theta$. The probability is defined to be the product over the tree $\bJ$, as in Section \ref{preliminaries}. The number giving the dimension can be defined using singular value functions instead of $\mR_\bi^s$. (For definition of singular value function, see \cite{Falconer88}, for example.) The main problem in proving Conjecture \ref{affine} follows from this; unlike in the similitude case, in the affine case the singular value function is not multiplicative, and multiplicativity is needed in multiple places in the proof of Theorem \ref{main}.

Let us next consider a somewhat more general, related problem. Let ${\bf T}=\{T_1,\dots,T_m\}$ be a collection of independent, $\theta$-distributed linear mappings, and fix $\ba=\{a_1,\dots ,a_m\}$, $a_i\in\mathbb R^d$ with $a_1\neq\dots\neq a_m$. Denote by $F({\bf T}, \ba)$ the limiting set corresponding to the IFS ${f_1,\dots, f_m}$, $f_i(x)=T_i(x) + a_i$. Then one can ask the question
\begin{question}\label{once}
Is it true that $\dimH F({\bf T}, \ba)=\dimH F ({\bf T})$, for almost all ${\bf T}$? 
\end{question}
This question is, in a sense, a more natural continuation of Falconer's \cite{Falconer88} than Conjecture \ref{affine}. However, it seems to be somewhat difficult to verify the answer one way or the other. 

\subsection*{Acknowledgements. }The author thanks Esa and Maarit J\"arvenp\"a\"a for many helpful comments during the preparation of the article, and Thomas Jordan and Pertti Mattila for useful remarks. 

\bibliography{ranaffbib}
\bibliographystyle{abbrv}

\end{document}